\documentclass[a4,12pt,reqno]{amsart}
\usepackage{amsmath,amsthm,amssymb}

\usepackage{color}

\theoremstyle{plain}
\newtheorem{thm}{Theorem}[section]
\newtheorem{lem}[thm]{Lemma}
\newtheorem{prop}[thm]{Proposition}

\newtheorem{ex}[thm]{Example}

  \makeatletter
  \@addtoreset{equation}{section}
  \makeatother

\title[Glaisher combinatorics]
{Glaisher combinatorics of regular partitions}
\author[Mizukawa and Yamada ]{Hiroshi Mizukawa
 and 
Hiro-Fumi Yamada}
\thanks{The first  author was supported by KAKENHI 24740033. The second author was supported by KAKENHI 24540020. }
\address{Hiroshi Mizukawa, Department of Mathematics,  National Defense
Academy of Japan,  Yokosuka 239-8686, Japan}
\email{mzh@nda.ac.jp}

\address{Hiro-Fumi Yamada, Department of Mathematics, Okayama University, Okayama 700-8530, 
Japan}
\email{yamada@math.okayama-u.ac.jp}
\date{}

\begin{document}
\maketitle
\begin{abstract}
Extending the notion of $r$-(class) regular partitions,
we define $(r_{1},\ldots,r_{m})$-class regular partitions.  
A partition identity is presented and described by making use of the 
Glaisher correspondence.   
\end{abstract}
\section{Introduction}
\noindent
Partitions of natural numbers are ubiquitous in representation theory.
Typically, they label the ordinary irreducible representations of the symmetric groups.
Turning to modular representations of the symmetric groups, some restrictions to the partitions 
naturally arise. Namely, for a prime $r$, $r$-modular irreducible representations are labeled by 
the $r$-regular partitions.
On the other hand, the $r$-regular conjugacy classes correspond to the $r$-class regular 
partitions. 
As Euler noticed, $r$-regular partitions of $n$ are equinumerous to the $r$-class regular 
partitions of $n$. The natural combinatorial bijection between these two sets is called the Glaisher
correspondence. 
One of the authors studied in \cite{asy} the graded version of Glaisher correspondence 
and revealed an intimate role of the correspondence in modular representation theory.
We feel that the ``Glaisher combinatorics" should possibly be one of 
the keys in the investigation of  the symmetric groups.

When we look at the $r$-modular ($r \geq 3$, odd)
 representations of the covering of the symmetric group.
We need to handle the partitions which are $2$-class regular and $r$-class regular.
In this note, motivated by the above, we define 
$\underline{r}$-regular / $\underline{r}$-class regular 
partitions for a mutually coprime integral sequence $\underline{r}=(r_{1},\ldots,r_{m})$.
We give some partition identities and generating functions.

The paper is organized as follow.
In Section 2, we derive our main  formula (Theorem \ref{t1}) which is on multiplicities 
of parts in $\underline{r}$-class regular partitions.
Section 3 is devoted to a rephrase of the formula in terms of the Glaisher 
correspondence. Although this is an easy algorithm, we expect this gives a path to 
representation theory mentioned above.
In Section 4, the case $\underline{r}=r$ is discussed.
We consider the $r$-regular character table of the symmetric group
and provide a proof of  Olsson's determinant formula \cite{bo,bos,ol}.
We examine the transition matrices of the Hall-Littlewood symmetric functions
and the Schur functions.

\section{$\underline{r}$-class regular partitions}
Let $\underline{r}=(r_{1},r_{2},\ldots,r_{m})$ be a tuple of positive integers $\geq 2$.
Throughout the paper, we assume that any two integers $r_{i}$ and $r_{j}$
  $(i\not= j)$ in $\underline{r}$ are coprime.
 If an integer $n$ is not divisible by $r_{1},r_{2},\ldots,r_{m}$, then we write
 $n \not \equiv 0 \pmod {\underline{r}}$.
 A partition $\lambda$ said to be $\underline{r}$-class regular if any parts 
 of $\lambda$ are not divisible by $r_{1},r_{2},\ldots,r_{m}$.
 Let $CP_{\underline{r},n}$ be the set of the $\underline{r}$-class regular partitions  
 of $n$.
 We put 
  $$\pi_{k}(q)=\prod_{1\leq i_{1}<\ldots <i_{k} \leq m}(1-q^{r_{i_{1}}\cdots r_{i_{k}}}).$$
and
$$
\Phi_{\underline{r}}(q)=
\begin{cases}
\displaystyle{\prod_{n \geq 1}\frac{\pi_{1}(q^n)\pi_{3}(q^n)\cdots \pi_{m-1}(q^n)}
{\pi_{2}(q^n)\pi_{4}(q^n)\cdots \pi_{m}(q^n)}\frac{1}{1-q^n}},& m \equiv 0\pmod{2}\\
\displaystyle{\prod_{n \geq 1}\frac{\pi_{1}(q^n)\pi_{3}(q^n)\cdots \pi_{m}(q^n)}
{\pi_{2}(q^n)\pi_{4}(q^n)\cdots \pi_{m-1}(q^n)}\frac{1}{1-q^n}},& m \equiv 1\pmod{2}.
\end{cases}
$$
Then the inclusion-exclusion principle gives us 
$$\Phi_{\underline{r}}(q)=\prod_{n \not \equiv 0 \pmod{\underline{r}}}\frac{1}{1-q^n}=\sum_{n \leq 0} |CP_{\underline{r},n}|q^n.$$
We define, for $j \geq 1$,
$$V_{\underline{r},j,n}=\sum_{\rho \in CP_{\underline{r},n}}m_{j}(\rho)
\ \ {\rm and}\ \ 
W_{\underline{r},j,n}=\sum_{\rho \in CP_{r,n}}\big|\{i \mid m_{i}(\rho) \geq j\}\big|
,$$
where $m_{i}(\rho)$ means the multiplicity of $i$ in $\rho$.
\begin{thm}\label{t1}
If  $j\not\equiv 0 \pmod{\underline{r}}$, then we have
$$V_{\underline{r},j,n}=\sum_{k_{1},\ldots,k_{m}\geq 0}
W_{\underline{r},r_{1}^{{k_{1}}}{r_{2}^{k_{2}}}\cdots r_{m}^{k_{m}}j,n}.$$
\end{thm}
Before proving this theorem, we give an example. 
\begin{ex}
We take $\underline{r}=(2,3)$ and $n=10$. 
There are four $(2,3)$-class regular partitions of n = 10.
 The following table lists $V_{\underline{r},j,n}$ and $W_{\underline{r},j,n}$ of them:
$$\begin{array}{c||cccccccccc}
j&1&2&3&4&5&6&7&8&9&10\\
\hline
V_{\underline{r},j,n}&18&0&0&0&3&0&1&0&0&0\\
W_{\underline{r},j,n}&6&4&3&2&2&1&1&1&1&1
\end{array}.$$
We have
$$
\begin{cases}
V_{\underline{r},1,10}&=
W_{\underline{r},1,10}+W_{\underline{r},2,10}+W_{\underline{r},3,10}
+W_{\underline{r},4,10}+
W_{\underline{r},6,10}+W_{\underline{r},8,10}
+W_{\underline{r},9,10},\\
V_{\underline{r},5,10}&=
W_{\underline{r},5,10}+W_{\underline{r},10,10},\\
V_{\underline{r},7,10}&=
W_{\underline{r},7,10}.
\end{cases}$$
\end{ex}
\begin{proof}
Let $j \not \equiv 0 \pmod{\underline{r}}$. We have
$$\Phi_{\underline{r}}(q)\frac{1-q^j}{1-tq^{j}}=\sum_{n \geq 0}\left(\sum_{\rho \in CP_{\underline{r},n}}t^{m_{j}(\rho)}\right)q^n.$$
Taking the $t$-derivative at $t=1$, we obtain 
\begin{equation}\label{1}
\Phi_{\underline{r}}(q)\frac{q^j}{1-q^{j}}=\sum_{n \geq 0}V_{\underline{r},j,n}q^n.
\end{equation}
Let $\ell \not \equiv 0 \pmod{\underline{r}}$ and $j \geq 1$. We have
\begin{align*}
\sum_{n \geq 0}|\{\rho \in CP_{\underline{r},n}\mid m_{k}(\rho)\geq j\}|q^n
&=\Phi_{\underline{r}}(q)(1-q^\ell)(q^{j\ell}+q^{(j+1)\ell}+q^{(j+2)\ell}+\cdots)\\
&=\Phi_{\underline{r}}(q)q^{j\ell}
\end{align*}
We take sum over $\ell \not \equiv 0 \pmod{\underline{r}}$ and obtain the generating function of 
$W_{\underline{r},j,n}$:
\begin{align*} 
&\sum_{n \geq 0}W_{\underline{r},j,n}q^n
=\Phi_{\underline{r}}(q)
\sum_{\ell\not\equiv 0\!\!\!\!\!\pmod{\underline{r}}}q^{j\ell}\\
&=\Phi_{\underline{r}}(q)\sum_{\ell \geq 1}\left\{q^{j\ell}
-\left(\sum_{i_{1}=1}^m q^{r_{i_{1}}j\ell}\right)
+\left(\sum_{1\leq i_{1}< i_{2} \leq m}
 q^{r_{i_{1}}r_{i_{2}}j\ell}\right)
-\cdots+(-1)^mq^{r_{1}r_{2}\cdots r_{m}j\ell}
 \right\}\\
 &=\Phi_{\underline{r}}(q)
 \left\{
  \frac{q^j}{1-q^j}+
\sum_{k=1}^{m}\left((-1)^k\sum_{1\leq i_{1}< i_{2}<\ldots<i_{k} \leq m}
\frac{q^{r_{i_{1}}r_{i_{2}}\cdots r_{i_{k}}j}}{1-q^{r_{i_{1}}r_{i_{2}}\cdots r_{i_{k}}j}}\right)
 \right\}.
\end{align*}
We replace $j$ by $r_{1}^{k_{1}}r_{2}^{{k_{2}}}\cdots r_{m}^{{k_{m}}}j$ in the above
and consider the sum
\begin{align}\label{4}
\sum_{n \geq 0}
\left(
\sum_{k_{1},\ldots,k_{m}\geq 0}
W_{\underline{r},r_{1}^{{k_{1}}}{r_{2}^{k_{2}}}\cdots r_{m}^{k_{m}}j,n}
\right)q^n.
\end{align}
Now we assume that  just $m' \ (0 \leq m' \leq m)$ entries of $(k_{1},\ldots,k_{m})$ are not zero.
Then we see  that the coefficient of $\frac{q^{r_{1}^{k_{1}}\cdots r_{m}^{k_{m}}j}}{1-q^{r_{1}^{k_{1}}\cdots r_{m}^{k_{m}}j}}$ is  $\binom{m'}{a}$  in 
$$\sum_{k_{1},\ldots,k_{m}\geq 0}
\left(
\sum_{1\leq i_{1}<\cdots <i_{a} \leq m}
\frac{q^{r_{i_{1}}r_{i_{2}}\cdots r_{i_{a}}r_{1}^{k_{1}}r_{2}^{{k_{2}}}\cdots r_{m}^{{k_{m}}}j}}{1-q^{r_{i_{1}}r_{i_{2}}\cdots r_{i_{a}}r_{1}^{k_{1}}r_{2}^{{k_{2}}}\cdots r_{m}^{{k_{m}}}j}}
\right).$$
Since $\sum_{a \geq0} (-1)^a\binom{m'}{a}=0$, we have 
\begin{align}\label{2}
\sum_{n \geq 0}\left(
\sum_{k_{1},\ldots,k_{m}\geq 0}
W_{\underline{r},r_{1}^{{k_{1}}}{r_{2}^{k_{2}}}\cdots r_{m}^{k_{m}}j,n}
\right)q^n=\Phi_{\underline{r}}(q)\frac{q^j}{1-q^j}.
\end{align}
From (\ref{1}) and (\ref{2}) we obtain the formula 
$$V_{\underline{r},j,n}=\sum_{k_{1},\ldots,k_{m}\geq 0}
W_{\underline{r},r_{1}^{{k_{1}}}{r_{2}^{k_{2}}}\cdots r_{m}^{k_{m}}j,n}.$$
\end{proof}
We put 
$$\begin{cases}
a_{\underline{r},n}=\prod_{\rho \in CP_{\underline{r},n}}\prod_{i=1}^{\ell(\rho)}\rho_{i}, \\
b_{\underline{r},n}=\prod_{\rho \in CP_{\underline{r},n}}\prod_{i\geq 1}m_{i}(\rho)!. 
\end{cases}$$
Then the following theorem holds.
\begin{thm}\label{23}
We have
$b_{\underline{r},n}=\prod_{i=1}^{m}r_{i}^{c_{r_{i},n}}a_{\underline{r},n},$
where $c_{{r_{i}},n}$ is given by
$$c_{{r}_{i},n}=\sum_{j \not\equiv 0 \!\!\!\!\! \pmod{\underline{r}}}\sum_{k_{1},\ldots,k_{m}\geq 0}
k_{i}\ W_{\underline{r},r_{1}^{{k_{1}}}{r_{2}^{k_{2}}}\cdots r_{m}^{k_{m}}j,n}.$$
\end{thm}
\begin{proof}
Since $V_{\underline{r},j,n}=0$ unless $j \not\equiv 0 \pmod{\underline{r}}$, we have
\begin{align*}
\prod_{\rho \in CP_{\underline{r},n}}\prod_{i \geq 1}\rho_{i}
=
\prod_{j \geq 1} j^{V_{\underline{r},j,n}}
=
\prod_{j \not\equiv 0\!\!\!\!  \pmod{\underline{r},}} j^{V_{\underline{r},,j,n}}.
\end{align*}
 We write $\underline{r}^{\underline{i}}=r_{1}^{i_1}\cdots r_{m}^{i_{m}}$.
On the other hand we compute
\begin{align*}
\prod_{\rho \in CP_{\underline{r},n}}\prod_{i \geq 1}m_{i}(\rho)!
&=
\prod_{j \geq 1} j^{W_{\underline{r},j,n}}
=
\prod_{j \not\equiv 0\!\!\!\!   \pmod{\underline{r}}} 
\prod_{i_1,\ldots,i_m \geq 0}(\underline{r}^{\underline{i}}j)^{W_{\underline{r},\underline{r}^{\underline{i}}j,n}}\\
&=
\left(
\prod_{j \not\equiv 0\!\!\!\!   \pmod{\underline{r}}} 
\prod_{i_1,\ldots,i_m \geq 0}(\underline{r}^{\underline{i}})^{W_{r,r^ij,n}}
\right)
\times
\left(
\prod_{j \not\equiv 0\!\!\!\!   \pmod{\underline{r}}} 
\prod_{i_1,\ldots,i_m \geq 0}j^{W_{r,r^ij,n}}
\right)\\
&=
\prod_{i=1}^{m}r_{i}^{c_{r_{i},n}}
\prod_{j \not\equiv 0\!\!\!\!   \pmod{r}}j^{\sum_{i_1,\ldots,i_m \geq 0}W_{r,\underline{r}^{\underline{i}}j,n}}\\
&\mathop{=}_{\rm Thm. \ref{t1}}
\prod_{i=1}^{m}r_{i}^{c_{r_{i},n}}
\prod_{j \not\equiv 0\!\!\!\!  \pmod{r}} j^{V_{\underline{r},j,n}}.
\end{align*}
%
%

\end{proof}

The generating functions of $c_{{r}_{i},n}$'s are given by the following theorem.
\begin{thm}
For $1\leq i \leq m$, we have
\begin{align*}
\sum_{n \geq 0} c_{r_{i},n}q^n
&=
\Phi_{\underline{r}}(q)\left(
\sum_{n \geq 0}
 \frac{q^{r_{i}n}}{1-q^{r_{i}n}}
 +
\sum_{k=1}^{m-1}(-1)^k
\sum_{\stackrel{1\leq l_{1}<\ldots< l_{k} \leq m}
{{l_{1},\ldots,l_{m}\not={i}}}}
\sum_{n\geq 0}\frac{q^{r_{i}r_{l_{1}}\cdots r_{l_{k}}n}}
{1-q^{r_{i}r_{l_{1}}\cdots r_{l_{k}}n}}
\right)\\
&=\Phi_{\underline{r}}(q)\sum_{n \not \equiv 0 \pmod{ \underline{r}^{(i)}}}\frac{q^{r_{i}n}}{1-q^{r_{i}n}},
\end{align*}
where $\underline{r}^{(i)}=(r_{1},\ldots,r_{i-1},r_{i+1},\ldots,r_{m})$.
\end{thm}
\begin{proof}
Without loss of generality,
we can assume $i=1$.
We compute
\begin{align*}
&\sum_{n\geq 0}\left(\sum_{i_{1},\ldots,i_{m}\geq 0} i_{1}W_{\underline{r},\underline{r}^{\underline{i}}j,n}\right)q^n
=
\sum_{i_{1},\ldots,i_{m}\geq 0}i_{1}\left( \sum_{n\geq 0}W_{\underline{r},\underline{r}^{\underline{i}}j,n}q^n\right)\\
&=
\Phi_{\underline{r}}(q)\sum_{i_{2},\ldots,i_{m}\geq 0}\sum_{i_{1}\geq 0}
i_{1}
\left\{
\frac{q^{\underline{r}^{\underline{i}}j}}{1-q^{\underline{r}^{\underline{i}}j}}
+
\sum_{k=1}^{m}(-1)^k\sum_{1\leq l_{1}<\ldots< l_{k} \leq m}\frac{q^{r_{l_{1}}\ldots r_{l_{k}}\underline{r}^{\underline{i}}j}}{1-q^{r_{l_{1}}\ldots r_{l_{k}}\underline{r}^{\underline{i}}j}}
\right\}\\
&=\Phi_{\underline{r}}(q)
\sum_{i_{2},\ldots,i_{m}\geq 0}
\left\{
\sum_{i_{1}\geq 1}i_{1}\left(
\frac
{q^{\underline{r}^{\underline{i}}j}}
{1-q^{\underline{r}^{\underline{i}}j}}
-
\frac{q^{r_{1}\underline{r}^{\underline{i}}j}}{1-q^{r_{1}\underline{r}^{\underline{i}}j}}\right)\right.\\
& \left.+\sum_{k=1}(-1)^k
\sum_{i_{1}\geq 1}i_{1}
\sum_{2\leq l_{1}<\ldots< l_{k} \leq m}\left(
\frac
{q^{r_{l_{1}}\ldots r_{l_{k}}\underline{r}^{\underline{i}}j}}
{1-q^{r_{l_{1}}\ldots r_{l_{k}}\underline{r}^{\underline{i}}j}}
-
\frac{q^{r_{1}r_{l_{1}}\ldots r_{l_{k}}\underline{r}^{\underline{i}}j}}{1-q^{r_{1}r_{l_{1}}\ldots r_{l_{k}}\underline{r}^{\underline{i}}j}}\right)
\right\}\\
&=\Phi_{\underline{r}}(q)
\sum_{i_{2},\ldots,i_{m}\geq 0}
\sum_{i_{1}\geq 1}
\left\{
\frac{q^{\underline{r}^{\underline{i}}j}}{1-q^{\underline{r}^{\underline{i}}j}}
+
\sum_{k=1}^{m}(-1)^k
\sum_{2\leq l_{1}<\ldots< l_{k} \leq m}\frac{q^{r_{l_{1}}\ldots r_{l_{k}}\underline{r}^{\underline{i}}j}}{1-q^{r_{l_{1}}\ldots r_{l_{k}}\underline{r}^{\underline{i}}j}}
\right\}.
\end{align*}
Now we take sum over $j \not \equiv 0 \pmod{\underline{r}}$ to have the generating 
function of $c_{r_{1},n}$ as desired.
The second equality in the theorem follows from the inclusion-exclusion principle.
\end{proof}
\section{Glaisher Combinatorics}
Let $RP_{\underline{r},n}$ be the set of partitions whose parts are not divisible by $r_{i}$
for any $i=2,3,\ldots,m$ and 
the multiplicity of each part is less than $r_{1}$.
A partition $\lambda \in RP_{\underline{r},n}$ said to be an $\underline{r}$-regular.
We rewrite $\Phi_{\underline{r}}(q)$ as follows:
\begin{align*}
&\Phi_{\underline{r}}(q)
=\prod_{n \geq 1}\frac{1-q^{r_{1}n}}{1-q^n}
\prod_{i=2}^{m}\frac{(1-q^{r_i n})}{(1-q^{r_{1}r_{i}n})}
\prod_{i<j}\frac{(1-q^{r_{1}r_i r_{j}n})}{(1-q^{r_{i}r_{j}n})}\cdots\\
&=
\prod_{n \geq 1}\left(\sum_{k=0}^{r_{1}-1}q^{nk}\right)
\prod_{i=2}^{m}\frac{1}{
\sum_{k=0}^{r_{1}-1}q^{r_{i}nk}
}
\prod_{i<j}
\left(\sum_{k=0}^{r_{1}-1}q^{r_{i}r_{j}nk}\right)
\cdots\\
&=\sum_{n \geq 0} |RP_{\underline{r},n}|q^n
\end{align*}
Therefore we have $|CP_{\underline{r},n}|=|RP_{\underline{r},n}|$.
A concrete bijection will be described in this section.  
The following proposition is a direct consequence of the generating function $\Phi_{\underline{r}}(q)$.
\begin{prop}
For any permutation $\underline{s}=(s_{1},\ldots,s_{m})$ of
$\underline{r}=(r_{1},\ldots,r_{m})$, we have
$$|RP_{\underline{r},n}|=|RP_{\underline{s},n}|.$$
\end{prop}
For example, the number of 2-regular, 3-class regular partitions of $n$ is equal to the number of
3-regular, 2-class regular partitions of $n$.
This is also equal to the number of partitions of $n$ whose parts are of the form $6k\pm1\ (k \geq 0)$.

There is a natural bijection between the sets $RP_{\underline{r},n}$
and $CP_{\underline{r},n}$.
Take $\lambda \in RP_{\underline{r},n}$. If $\lambda$
 has a  multiple of $r_{1}$
 as a part, say $kr_{1}$, then replace $kr_{1}$ by
 $k^{r_{1}}$. 
 By this step the length of the partition increases by $r-1$.
 Repeat these steps until the partition has come to an element $g_{r_{1}}(\lambda)$ 
 of $CP_{\underline{r},n}$. 
 The map $g_{r_{1}}: RP_{\underline{r},n} \rightarrow
 CP_{\underline{r},n}$ is called the Glaisher correspondence, and shown to be bijective.
 
 The number of steps for obtaining $g_{r_{1}}(\lambda) \in CP_{\underline{r},n}$
 from $\lambda \in RP_{\underline{r},n}$ equals 
 $$\frac{\ell(g_{r_{1}}(\lambda))-\ell(\lambda)}{r-1}.$$
 Define for $\lambda \in CP_{\underline{r},n}$ and $j \not\equiv 0 \pmod{r_{1}}$,
 $y_{r^{k}j}=|\{i \geq 1 \mid m_{i}(\rho) \geq r^{k}j\},|$
 and 
 $$G_{j}(\lambda)=\sum_{k \geq 1}k y_{r_{1}^{k}j}(\lambda).$$
 For example, if $r_{1}=3$ and $\lambda=(1^9)$, then 
 $G_{1}(\lambda)=3, G_{2}(\lambda)=1$ and $G_{j}(\lambda)=0$ otherwise.
 Put $G(\lambda)=\sum_{j\not \equiv 0 \pmod{r_{1}}}G_{j}(\lambda)$.
 This is nothing but the times of Glaisher steps for $g^{-1}(\lambda) \mapsto \lambda$,
 and also we have the following.
 \begin{prop}
 $$c_{r_{1},n}=\sum_{\lambda \in CP_{\underline{r},n}}G(\lambda).$$
 \end{prop}
 \begin{proof}
Proof is just by interchanging the order of the summation.
 \end{proof}

\section{case of $\underline{r}=r$ and regular character table}
%
Throughout this section, we fix a positive integer $r \geq 2$.
Here we restrict our attention to the case $m=1$, i.e., $\underline{r}=r$.
We will relate the analysis of $r$-(class) regular partitions
with the character tables of the symmetric groups.
 It should be remarked  that \cite{bos}
already proved Theorem \ref{main}. They give a bijective proof, and also sketch a proof using generating 
functions. Here we supply the proof relying on the generating functions for the sake of completeness.

For a partition $\lambda=(\lambda_{1},\lambda_{2},\ldots)$ and $j \in \{1,2,\ldots,r-1\}$, we put
$$
x_{r,j}(\lambda)
=\big|\{i \mid \lambda_{i} 
\equiv j \pmod{r}\}\big|
\ {\rm and }\ \  
y_{r,j}(\lambda)=\big|\{i \mid m_{i}(\lambda) \geq j\}\big|.$$ 
We define
$$X_{r,j,n}=\sum_{\rho \in CP_{r,n}}x_{r,j}(\rho)\ {\rm and} \ \  Y_{r,j,n}=\sum_{\lambda \in RP_{r,n}}y_{r,j}(\lambda).$$
\begin{thm}\label{main}
$
X_{r,j,n}
-
Y_{r,j,n}=c_{r,n}$
for $j=1,2,\ldots,r-1$.
\end{thm}
Before proving this theorem, we give an example. 
\begin{ex}
We take $r=3$ and $n=7$. 
The following table lists the $3$-class regular partitions of $n=7$:
$$\begin{array}{c||ccccccccc|c}
\rho&7&52&51^2&421&41^3&2^31&2^21^3&21^5&1^7&{\rm total}\\
\hline
x_{3,1}(\rho)&1&0&2&2&4&1&3&5&7&25\\
x_{3,2}(\rho)&0&2&1&1&0&3&2&1&0&10
\end{array}$$
From the table above, we have $X_{3,1,7}=25$ and $X_{3,2,7}=10$.
 As for the $3$-regular partitions of $7$, we have
$$\begin{array}{c||ccccccccc|c}
\lambda&7&61&52&51^2&52&421&3^21&32^2&321^2&{\rm total}\\
\hline
\prod_{i \geq 1} m_{i}(\lambda)!&1&1\!\cdot\!1&1\!\cdot\!1&1\!\cdot\!21&1\!\cdot\!1&1\!\cdot\!1\!\cdot\!1&21\!\cdot\!1&1\!\cdot\!21&1\!\cdot\!1\!\cdot\!21&-\\
y_{3,1}(\lambda)&1&2&2&2&2&3&2&2&3&19\\
y_{3,2}(\lambda)&0&0&0&1&0&0&1&1&1&4
\end{array}$$
From the second table, we have $Y_{3,1,7}=19$ and $Y_{3,2,7}=4$.
Thus we see 
$$X_{3,1,7}-Y_{3,1,7}=X_{3,2,7}-Y_{3,2,7}=6.$$
On the other hand  we have
$$\Phi_3(q)  \sum_{k\geq 1}\frac{q^{3k}}{1-q^{3k}}=q^3 + q^4 + 2 q^5 + 4 q^6 + 6 q^7 + 9 q^8 + 13 q^9 + 19 q^{10}+\cdots.$$
\end{ex}
\begin{proof}
First we will compute the generating function of $X_{r,j,n}$.
For $i \not \equiv 0 \pmod{r}$, we have
$$\Phi_{r}(q)\frac{1-q^i}{1-tq^i}=\sum_{n\geq 0}\left(t^{\sum_{\rho \in CP_{r,n}}m_{i}(\rho) }\right)q^n.$$
Taking the $t$-derivative at $t=1$, we obtain
\begin{align}\label{V}
\Phi_{r}(q)\frac{q^i}{1-q^i}=\sum_{n\geq 0} \left(\sum_{\rho \in CP_{r,n}}m_{i}(\rho) \right)q^n.
\end{align}
Since $x_{r,j}(\rho)=\sum_{k \geq 0}m_{kr+j}(\rho)$,
we have the following generating function 
 of $X_{r,j,n}$.
$$
\sum_{n\geq 0}X_{r,j,n}q^n=\Phi_{r}(q)\sum_{k \geq 0}\frac{q^{rk+j}}{1-q^{rk+j}}.$$
Second, we consider the $r$-regular partitions and the generating function of $Y_{r,j,n}$.
We put 
\begin{equation}\label{phj}
\Phi_{r,j}(q,t)=\prod_{k \geq 1} (1+q^k+q^{2k}+\cdots+q^{(j-1)k}+tq^{jk}+tq^{(j+1)k}+\cdots+tq^{(r-1)k}).
\end{equation}
Immediately we have
\begin{align*}
\Phi_{r,j}(q,t)&=\sum_{n \geq 0}\left(\sum_{\lambda \in RP_{r,n}}t^{y_{r,j}(\lambda)}\right)q^n.
\end{align*}  
Taking the $t$-derivative at $t=1$, we obtain
$$\frac{d}{dt}\Phi_{r,j}(q,t)\Big|_{t=1}=\sum_{n \geq 0}
\left(
\sum_{\lambda \in RP_{r,n}}{y_{r,j}(\lambda)}
\right)
q^n
=\sum_{n \geq 0} 
Y_{r,j,n}
q^n.
$$
As for the equation (\ref{phj}), we have
$$\frac{d}{dt}\Phi_{r,j}(q,t)\Big|_{t=1}=
\Phi_{r,j}(q,t)\sum_{k \geq 1}\frac{q^{jk}-q^{rk}}{1-q^{rk}}.$$
The generating function 
 of $Y_{r,j,n}
$ reads
$$
\sum_{n \geq 0}
Y_{r,j,n}
q^n
=
\Phi_{r}(q)\sum_{k \geq 1}\frac{q^{jk}-q^{rk}}{1-q^{rk}}.$$
To complete the proof, we compute
\begin{align*}
\sum_{n \geq 0}
X_{r,j,n}
q^n-\sum_{n \geq 0}
Y_{r,j,n}
q^n&=\Phi_{r}(q)\left(\sum_{k \geq 0}(q^{rk+j}+q^{2(rk+j)}+q^{3(rk+j)}+\cdots) \right.\\
&-\left.\sum_{m \geq 1}(q^{jm}+q^{(r+j)m}+q^{(2r+j)m}+\cdots)+\sum_{k\geq1}\frac{q^{rk}}{1-q^{rk}}\right)\\
&=\Phi_{r}(q)\left(\sum_{k \geq 0}\sum_{m\geq 1}q^{m(rk+j)} 
-\sum_{m \geq 1}\sum_{k\geq 0}q^{(kr+j)m}+\sum_{k\geq 1}\frac{q^{rk}}{1-q^{rk}}\right)\\
&=\Phi_{r}(q)\sum_{k \geq 1}\frac{q^{rk}}{1-q^{rk}}=\sum_{n\geq 0}c_{r,n}q^n.
\end{align*}
\end{proof}
\subsection{Hall-Littlewood symmetric functions at root of unity}
Next, we apply Theorem \ref{main} to computations
 of some minor determinants of 
transition matrices and the character tables of the symmetric groups.
The Hall-Littlewood $P$- and $Q$- symmetric functions (\cite{mac}) are
 a one parameter family of symmetric functions 
satisfying the orthogonality relation:
$$\langle P_{\lambda}(x;t), Q_{\mu}(x;t)\rangle_{t}=\delta_{\lambda\mu},$$
where the inner product $\langle , \rangle_{t}$ is defined by $\langle p_{\lambda}(x) ,
p_{\mu}(x) \rangle_{t}=z_{\lambda}(t)\delta_{\lambda\mu}$ 
with $z_{\lambda}(t)=z_{\lambda}\prod_{i \geq 1}(1-t^{\lambda_{i}})^{-1}$.
Let $(a;t)_{n}$ be a $t$-shifted factorial:
 $$(a;t)_{n}=
 \begin{cases}
  (1-a)(1-at)\cdots(1-at^{n-1})&(n \geq 1)\\
  1 & (n=0).
 \end{cases}$$
The relation between $P$- and $Q$- functions is described as
$$Q_{\lambda}(x)=b_{\lambda}(t)P_{\lambda}(x),$$
where $b_{\lambda}(t)=\prod_{i \geq 1}(t;t)_{m_{i}(\lambda)}$.
\subsection{$Q'$-functions}
We are interested in the case that parameter $t$ is a primitive $r$-th root of unity $\zeta$.
The Hall-Littlewood symmetric functions at root of unity is studied 
at the first time by \cite{mo}.
We remark that $\{Q_{\lambda}(x;\zeta)\mid \lambda \in RP_{r,n}\}$
 is a ${\mathbb Q}(\zeta)$-basis for the subspace $\Lambda^{(r)}={\mathbb Q}(\zeta)[p_{s}(x)\mid s\not\equiv 0 \pmod{r}]$ 
of the symmetric function ring $\Lambda={{\mathbb Q}(\zeta)}[p_{s}(x)\mid s=1,2,\ldots]$.
This can be shown along the arguments in \cite[Chap. 3-8]{mac},
where the case $r=2$ is discussed.
In \cite{llt}, Lascoux, Leclerc and Thibon consider the dual basis $(Q'_{\lambda})$ of $P$-functions, relative 
to the inner product at $t=0$. Namely $P$- and $Q'$- functions satisfy the  
Cauchy identity:
$$\sum_{\lambda}P_{\lambda}(x;t)Q'_{\lambda}(y;t)=\prod_{i,j}(1-x_{i}y_{j})^{-1}.$$
When $t=\zeta$, the $Q'$-functions have the following nice factorization property.
\begin{prop}[\cite{llt}]\label{facQ}
Let $\zeta$ be a primitive $r$-th root of unity.
If a partition $\lambda$ satisfies $m_{i}(\lambda) \geq r$, then we have
$$Q'_{\lambda}(x;\zeta)=(-1)^{i(r-1)}Q'_{\lambda \setminus (i^r)}(x;\zeta)h_{i}(x^r).$$
Here $h_{i}(x^r)=h_{i}(x_{1}^r,x_{2}^r,\ldots)$ and $\lambda \setminus (i^r)$ is 
a partition obtained by removing the rectangle $(r^i)$  from the Young diagram $\lambda$.
\end{prop}
We define an $r$-{\it reduction} for a symmetric function $f(x)$ by
$$f^{(r)}(x)=f(x)\big|_{p_{r}(x)=p_{2r}(x)=p_{3r}(x)=\ldots=0}.$$
Proposition \ref{facQ} leads us to  the following lemma.
\begin{lem}\label{van}
${Q'}^{(r)}_{\lambda}(x;\zeta)=0$ unless $\lambda$ is an $r$-regular partition.
\end{lem}
We set
$$
Q_{\lambda}(x;\zeta)=\sum_{\rho \in CP_{r,n}}Q_{\rho}^{\lambda}p_{\rho}(x)
\ \ {\rm and}\ \  
{Q'}^{(r)}_{\lambda}(x;\zeta)=\sum_{\rho \in CP_{r,n}}{Q'}_{\rho}^{\lambda}p_{\rho}(x).
$$ 
\begin{prop}\label{QQ}
Let $\lambda \in RP_{r,n}$ and $\rho \in CP_{r,n}$. We have
$${Q'}_{\rho}^{\lambda}=\prod_{i \geq 1}(1-\zeta^{\rho_{i}})^{-1}Q_{\rho}^{\lambda}.$$
\end{prop}
\begin{proof}
We compute  inner products at $t=\zeta$ and $t=0$
for $r$-regular partitions $\lambda$ and $\mu$.
Namely, we see
\begin{align*}\delta_{\lambda \mu}=\langle P_{\lambda}(x;\zeta)
,Q_{\mu}(x;\zeta)\rangle_{\zeta}
=b_{\lambda}(\zeta)^{-1}\sum_{\rho\in CP_{r,n}}Q_{\rho}^{\lambda}Q_{\rho}^{\mu}z_{\rho}(\zeta)
\end{align*}
and
\begin{align*}\delta_{\lambda \mu}=\langle P_{\lambda}(x;\zeta)
,{Q'}^{(r)}_{\mu}(x;\zeta)\rangle_{0}
=b_{\lambda}(\zeta)^{-1}\sum_{\rho\in CP_{r,n}}Q_{\rho}^{\lambda}{Q'}_{\rho}^{\mu}z_{\rho}.
\end{align*}
Since $\{P_{\lambda}(x;\zeta) \mid \lambda \in RP_{r,n}\}$ is also a basis of $\Lambda^{(r)}$, we have the 
claim.
\end{proof}
We define $L_{\lambda\mu}(t)$ by
$$s_{\lambda}(x)=\sum_{\mu \in P_{n} }L_{\lambda\mu}(t)Q'_{\mu}(x;t),$$
where $s_{\lambda}(x)$ denotes the Schur function.
Let  $K_{\lambda\mu}(t)$ be the Kostka-Foulkes polynomial (\cite{mac}).
In other words, the matrix $K(t)=(K_{\lambda\mu}(t))_{\lambda,\mu \in P_{n}}$ is the transition matrix $M(s,P)$ from 
the Schur functions to the Hall-Littelewood $P$-functions. It is known that $K(t)$ is an upper unitriangular 
matrix.
\begin{lem}\label{kos} For  partitions $\lambda$ and $\mu$, we have 
$L_{\lambda\mu}(t)=K_{\mu\lambda}^{(-1)}(t)$, 
 the $(\lambda,\mu)$-entry  of the matrix $K(t)^{-1}$. 
\end{lem}
\begin{proof}
$$L_{\lambda\mu}(t)=\langle s_{\lambda}(x),P_{\mu}(x;t)\rangle_{0}=\langle s_{\lambda}(x),
\sum_{\nu\in P_{n}} K_{\mu\nu}^{(-1)}s_{\nu}\rangle_{0}=K_{\mu\lambda}^{(-1)}(t).$$
\end{proof}
\begin{ex}
We take $\zeta=-1$ and $n=4$. Then we have
\begin{align*}
s_{4}(x)&={Q'}_{4}(x;-1),\\
s_{31}(x)&={Q'}_{31}(x;-1)+{Q'}_{4}(x;-1),\\
s_{22}(x)&={Q'}_{22}(x;-1)+{Q'}_{31}(x;-1),\\
 s_{211}(x)&={Q'}_{211}(x;-1)+{Q'}_{22}(x;-1)+{Q'}_{31}(x;-1)+{Q'}_{4}(x;-1),\\
s_{1111}(x)&={Q'}_{1111}(x;-1)+{Q'}_{211}(x;-1)-{Q'}_{22}(x;-1)+{Q'}_{4}(x;-1).
\end{align*}
\end{ex}
Lemma \ref{van} and \ref{kos} give the following expansion formula.
\begin{prop}\label{lt} 
Let $\lambda \in P_{n}$  and $\mu \in RP_{r,n}$. We have
$$s^{(r)}_{\lambda}(x)=\sum_{\mu \in RP_{r,n} }K^{(-1)}_{\mu\lambda}(\zeta){Q'}^{(r)}_{\mu}(x;\zeta).$$ 
In particular, $(L_{\lambda\mu}(\zeta))_{\lambda,\mu \in RP_{r,n}}$ is a lower unitriangular matrix.
\end{prop}
\begin{ex}
By Proposition \ref{lt}, we immediately see
\begin{align*}
s^{(2)}_{4}(x)&={Q'}^{(2)}_{4}(x;-1),\\
s^{(2)}_{31}(x)&={Q'}^{(2)}_{31}(x;-1)+{Q'}^{(2)}_{4}(x;-1),\\
s^{(2)}_{22}(x)&={Q'}^{(2)}_{31}(x;-1),\\
 s^{(2)}_{211}(x)&={Q'}^{(2)}_{31}(x;-1)+{Q'}^{(2)}_{4}(x;-1),\\
s^{(2)}_{1111}(x)&={Q'}^{(2)}_{4}(x;-1).
\end{align*}
From the first two equations above, we have
$$(L_{\lambda\mu}(-1))_{\lambda,\mu \in RP_{2,4}}
=\begin{pmatrix}
1&0\\
1&1
\end{pmatrix}.$$
\end{ex}
We set 
$$s^{(r)}=\{s^{(r)}_{\lambda}(x)\mid \lambda \in RP_{r,n}\}, \ 
{Q'}^{(r)}=\{{Q'}^{(r)}_{\lambda}(x)\mid \lambda \in RP_{r,n}\}$$ and 
$$p^{(r)}=\{p_{\lambda}(x)\mid \lambda \in CP_{r,n}\}.$$
For $u,v \in \{s^{(r)},{Q'}^{(r)},p^{(r)}\}$,
we denote by $M(u,v)$ the transition matrix from $u$ to $v$.
By Lemma \ref{van}, we have that $M(s^{(r)},{Q'}^{(r)})$ is obtained by removing non $r$-regular rows and 
columns from the transposed inverse of $K(t)$.
\begin{thm}\label{Qp}
We have $\det M(Q'^{(r)},p^{(r)}) \in {\mathbb R}$ and
$$\det M(Q'^{(r)},p^{(r)})=\pm\frac{1}{r^{c_{r,n}}
{{\prod_{\rho \in CP_{r}(n)} \prod_{i \geq 1}\rho_{i}}}}.
$$
\end{thm}
\begin{proof}
The orthogonality relation of $P_{\lambda}(x;\zeta)$ and $Q'_{\mu}(x;\zeta)$:
\begin{align*}\delta_{\lambda \mu}=\langle P_{\lambda}(x;\zeta)
,{Q'}^{(r)}_{\mu}(x;\zeta)\rangle_{0}
=b_{\lambda}(\zeta)^{-1}\sum_{\rho\in CP_{r,n}}Q_{\rho}^{\lambda}{Q'}_{\rho}^{\mu}z_{\rho},
\end{align*}
and Proposition \ref{QQ} give 
\begin{align*}
\det M(Q'^{(r)},p^{(r)})^2&=\left(\prod_{\rho \in CP_{r,n}}\prod_{i \geq 1}\frac{1}{1-\zeta^{\rho_{i}}}
\right)^2
\frac{\prod_{\lambda \in RP_{r,n}}b_{\lambda}(\zeta)}{\prod_{\rho \in CP_{r,n}}z_{\rho}(\zeta)}\\
&=
\frac{\prod_{\lambda \in RP_{r,n}}b_{\lambda}(\zeta)}
{\prod_{\rho \in CP_{r,n}}z_{\rho}(\prod_{\rho \in CP_{r,n}}\prod_{i \geq 1}{(1-\zeta^{\rho_{i}}))}}\\
&=
\frac{\prod_{\lambda \in RP_{r,n}}\prod_{i \geq 1}(1-\zeta^{m_{i}(\lambda)})}
{\prod_{\rho \in CP_{r,n}}z_{\rho}(\prod_{\rho \in CP_{r,n}}\prod_{i \geq 1}{(1-\zeta^{\rho_{i}}))}}\\
&=
\frac{\prod_{\lambda \in RP_{r,n}}\prod_{j=1}^{r-1}(1-\zeta^{j})^{y_{r,j}(\lambda)}}
{\prod_{\rho \in CP_{r,n}}z_{\rho}(\prod_{\rho \in CP_{r,n}}\prod_{j = 1}^{r-1}{(1-\zeta^{j})^{x_{r,j}(\rho)})}}\\
&=
\frac{\prod_{j=1}^{r-1}(1-\zeta^{j})^{Y_{r,j,n}}}
{\prod_{\rho \in CP_{r,n}}z_{\rho}\prod_{j = 1}^{r-1}{(1-\zeta^{j})^{X_{r,j,n}}}}\\
&=
\frac{1}
{\prod_{\rho \in CP_{r,n}}z_{\rho}\prod_{j = 1}^{r-1}{(1-\zeta^{j})^{X_{r,j,n}-Y_{r,j,n}}}}\\
&=
\frac{1}
{\prod_{\rho \in CP_{r,n}}z_{\rho}\prod_{j = 1}^{r-1}{(1-\zeta^{j})^{c_{r,n}}}}.
\end{align*}
We apply Theorem \ref{main} to the last equality above.
By noticing $\prod_{i=1}^{r-1}(1-\zeta^{i})=r$ and using Theorem \ref{23}, we obtain the formula. 
\end{proof}
\subsection{Regular character tables of the symmetric groups}
Let $T_{n}=(\chi_{\rho}^{\lambda})_{\lambda,\rho \in P_{n}}$ be the ordinary character table of the symmetric group $S_{n}$.
The orthogonality relation of the characters implies
$$(\det T_{n})^2=\prod_{\rho \in P_{n}}z_{\rho} .$$
From James's book \cite[Corollary 6.5]{jam}, this formula can be simplified as
$$(\det T_{n})^2=\prod_{\rho \in P_{n}}\prod_{i \geq 1}\rho_{i}^2.$$
Olsson considers the $r$-regular character table  
$T^{(r)}_{n}=
(\chi_{\rho}^{\lambda})_{\substack{\lambda \in RP_{r,n}\\ 
\rho \in CP_{r,n}}}$ and computes its determinant. He proves the following theorem.
\begin{thm}[\cite{ol}]\label{osn1}
$$\det T^{(r)}_{n}=\pm \prod_{\rho \in CP_{r,n}}\prod_{i \geq 1}\rho_{i}.$$
\end{thm}
\begin{proof}
Theorem \ref{Qp} and Proposition \ref{lt}
enable us to compute the determinant of the regular character table as follows:
\begin{align*}
\det M(s^{(r)},{p}^{(r)})^2&=\det M(s^{(r)},{Q'}^{(r)})^2\det M({Q'}^{(r)},p^{(r)})^2\\
&=1\times \frac{1}
{r^{2c_{r,n}}
{{\prod_{\rho \in CP_{r,n}} \prod_{i \geq 1}\rho_{i}^2}}}\end{align*}
Since 
\begin{align*}\det M(s^{(r)},{p}^{(r)})^2=(\det T_{n}^{(r)})^2 \times 
{{\prod_{\rho \in CP_{r,n}} z_{\rho}^{-2}}}
\mathop{=}_{\rm Thm \ref{23}}(\det T_{n}^{(r)})^2 \times r^{-2c_{r,n}}
{{\prod_{\rho \in CP_{r,n}} \prod_{i\geq1}\rho_{i}^{-4}}},
\end{align*}
we have
$$(\det T_{n}^{(r)})^2=\left(
\prod_{\rho \in CP_{r,n}} \prod_{i\geq1}\rho_{i}
\right)^2.$$
\end{proof}
%
%
%
%
%

%

\end{document}